\documentclass{amsart}

\usepackage[utf8]{inputenc}\usepackage{tgpagella}\usepackage{mathpazo}

\usepackage{sty/Donald}

\usepackage{sty/Alphabets}      
\usepackage{sty/Definitions}    
\usepackage{sty/PageSetup}      
\usepackage{sty/Environments}   

\begin{document}
\title{A bicategorical pasting theorem}
\author{Niles Johnson}
\address{Department of Mathematics\\
	 The Ohio State University at Newark\\
	 1179 University Drive\\ 
	 Newark, OH 43055, USA}
       \email{johnson.5320@osu.edu\\
       yau.22@osu.edu}
\author{Donald Yau}


\begin{abstract}
We provide an elementary proof of a bicategorical pasting theorem that does not rely on Power's $2$-categorical pasting theorem, the bicategorical coherence theorem, or the local characterization of a biequivalence.
\end{abstract}



\date{02 October, 2019}
\maketitle




\section{Introduction}\label{sec:intro}

Bicategories and their pasting diagrams were introduced by B\'{e}nabou
\cite{benabou}.  Pasting diagrams in bicategories,
such as the following,
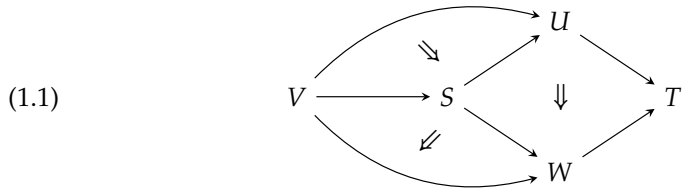
\begin{equation}\label{pasting-example}
\begin{tikzpicture}[x=20mm,y=20mm,baseline=(v.base)]
  \newcommand{\core}{
    \draw[0cell] 
    (0,0) node (v) {V}
    (1,0) node (s) {S}
    (1.75,.5) node (u) {U}
    (1.75,-.5) node (w) {W}
    (2.5,0) node (t) {T}
    ;
    \draw[1cell] 
    (s) edge[swap] node {} (u)
    (s) edge node {} (w)
    (u) edge node {} (t)
    (w) edge[swap] node {} (t)
    ;
    \draw[2cell] 
    node[between=s and t at .5, rotate=-90,font=\Large] (T2) {\Rightarrow}
    (T2) node[right] {}
    ;
  }
    \core
    \draw[1cell] 
    (v) edge node[pos=.4] {} (s)
    (v) edge[bend left] node (f1) {} (u)
    (v) edge[bend right, swap] node (g1) {} (w)
    ;
    \draw[2cell] 
    node[between=f1 and s at .5, rotate=-45,font=\Large] (T1) {\Rightarrow}
    (T1) node[above right] {}
    node[between=g1 and s at .5, rotate=225,font=\Large] (T3) {\Rightarrow}
    (T3) node[below right] {}
    ;
\end{tikzpicture}
\end{equation}
are analogous to commutative diagrams
in categories.  They allow one to use diagrams to express iterated
vertical composites of $2$-cells of the form
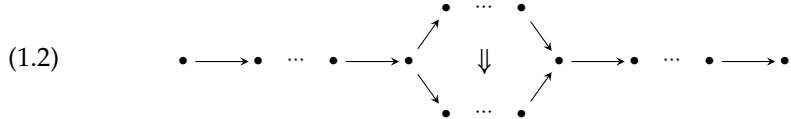
\begin{equation}\label{generic-atomic-graph}
\begin{tikzpicture}[scale=1, shorten >=-3pt, shorten <=-2pt,baseline=(a1.base)]
\tikzset{every node/.style={scale=.8, inner sep=2mm}} 
\node (a1) at (0,0) {$\bullet$}; \node (a2) at (1,0) {$\bullet$}; \node at (1.5,0) {$\cdots$};
\node (a3) at (2,0) {$\bullet$}; \node (a4) at (3,0) {$\bullet$};
\node (x1) at (3.5,.7) {$\bullet$}; \node at (4,.7) {$\cdots$}; \node (x2) at (4.5,.7) {$\bullet$};
\node (y1) at (3.5,-.7) {$\bullet$}; \node at (4,-.7) {$\cdots$}; \node (y2) at (4.5,-.7) {$\bullet$};
\node (b1) at (5,0) {$\bullet$}; \node (b2) at (6,0) {$\bullet$}; \node at (6.5,0) {$\cdots$};
\node (b3) at (7,0) {$\bullet$}; \node (b4) at (8,0) {$\bullet$};
\node at (4,0) {\rotatebox{270}{\LARGE{$\Rightarrow$}}}; 
\draw [arrow] (a1) to (a2); \draw [arrow] (a3) to (a4);
\draw [arrow] (a4) to (x1); \draw [arrow] (a4) to (y1);
\draw [arrow] (x2) to (b1); \draw [arrow] (y2) to (b1);
\draw [arrow] (b1) to (b2); \draw [arrow] (b3) to (b4);
\end{tikzpicture}
\end{equation}
with some bracketings of the top and bottom paths that are compatible with the (co)domain of the middle $2$-cell.  For example, the two triangle identities that define an internal adjunction in a bicategory can be compactly expressed in terms of a few pasting diagrams; see \cite{kelly-street} (Section 2.1) for the $2$-category case.  Moreover, the definitions of monoidal bicategories, as well as their symmetric, sylleptic, and braided variants, involve a number of large pasting diagrams \cite{mccrudden,stay}, without which the long vertical composites would be very hard to read.

A pasting theorem asserts that each pasting diagram has a uniquely defined composite that is independent of the order of the vertical composites, as long as they are defined.  For $2$-categories, such a pasting theorem was proved by Power \cite{power}.  There are basically two steps.  First he defined a concept of graphs with an acyclicity condition that ensures the existence of a composite in a $2$-category.  Then he showed by an induction that this composite has the desired uniqueness property.

For general bicategories, Verity \cite{verity} proved a bicategorical pasting theorem by extending Power's concept of graphs to include bracketings of the (co)domain of each interior face and of the  global (co)domain.  Briefly, he first applied the Bicategorical Coherence Theorem \cite{maclane-pare,street_cat-structures}, which asserts each bicategory $\B$ is retract biequivalent to a $2$-category $\A$.  With such a biequivalence $h : \B \to \A$, a pasting diagram in $\B$ yields a pasting diagram in $\A$, which has a unique composite by Power's pasting theorem for $2$-categories.  Using the fact that $h$ is locally full and faithful, a unique $2$-cell composite is then obtained back in $\B$.  The proof that this composite is independent of the choice of a biequivalence $h$ also relies on the Bicategorical Coherence Theorem.

The purpose of this paper is to prove a bicategorical pasting theorem that does \emph{not} rely on (i) Power's $2$-categorical pasting theorem, (ii) the Bicategorical Coherence Theorem, (iii) the local characterization of a biequivalence, or (iv) that $\Bicat(\B,\B)$ is a bicategory (with lax functors as objects, lax natural transformations as $1$-cells, and modifications as $2$-cells).  In fact, our proof stays entirely within the given bicategory, and only uses the basic axioms of a bicategory.

In addition to being much more elementary, our approach yields a $2$-categorical pasting theorem which is independent of Power's theorem.  Moreover, the authors were motivated by concurrent work \cite{JYquillen} to give a self-contained proof of the local characterization of biequivalences.  Pasting diagrams are an indispensible part of such work, and therefore one requires an independent pasting theorem.

The essential difference between a 2-categorical pasting theorem and a
bicategorical pasting theorem is the presence of nontrivial
associators.  One adds bracketings to specify the order of composition
of 1-cells, but then a bracketed pasting diagram does not necessarily
admit a composite.  For example, the unique bracketing of the diagram
in \cref{pasting-example} does not have a well-defined composite in a
general bicategory---one must extend the diagram by inserting
appropriate associators.  The content of this paper has three parts,
as follows.

First, in \cref{sec:pasting-schemes} we explain the graph theoretic
concepts necessary to understand pasting diagrams and their extensions
by associators; these are the notions of pasting scheme
(\cref{def:pasting-scheme}) and composition scheme
(\cref{def:bicategorical-pasting-scheme}).  The main result of this
section is \cref{bicat-pasting-existence}, which proves that every
pasting scheme extends to a composition scheme.

Second, in \cref{sec:pasting-diagrams} we apply the preceding graph
theory to explain pasting diagrams and their extensions to what we
call composition diagrams.  Every composition diagram has a
well-defined composite, as we detail in
\cref{def:bicat-diagram-composite}.  This section includes the
definition of bicategory to fix notation and terminology, together
with a detailed example for the diagram in \cref{pasting-example}.

Finally, in \cref{sec:bicat-pasting-theorem} we prove that the
composites resulting from any two extensions of a given pasting
diagram are equal.  This is the Bicategorical Pasting Theorem
  \ref{thm:bicat-pasting-theorem}.  Its proof depends on a
generalization of Mac Lane's Coherence Theorem, which we explain,
together with an induction argument similar to that of \cite{power}.
We note that, by restricting the argument to $2$-categories, we
recover a pasting theorem for $2$-categories which is essentially
Power's.

In the 2-categorical case, the only difference between our approach
and that of \cite{power,verity} is in our handling of the underlying
graph theory.  Power and Verity consider plane graphs with a source
and a sink, and bracketings in the bicategory case, that have no
directed cycles.  We also use plane graphs with a source and a sink,
and bracketings for all (co)domains.  However, instead of the
non-existence of directed cycles, our acyclicity condition is phrased
as the existence of a vertical decomposition of the graph into atomic
graphs, each containing one interior face like the one in
\cref{generic-atomic-graph}.  One advantage of this approach is that
it strictly mirrors the way pasting diagrams are usually used in
practice, namely, as vertical composites of 2-cells each produced by
whiskering a given $2$-cell with a number of $1$-cells.  Another is
that it greatly simplifies the graph theoretic work one must do,
particularly in the bracketed case.

\section{Pasting Schemes and Composition Schemes}\label{sec:pasting-schemes}

In this section we define the graph theoretic notions of pasting
scheme and composition scheme.
The main result of this section is \Cref{bicat-pasting-existence}.  It characterizes bracketed graphs that admit a composition scheme extension as those whose underlying anchored graphs admit a pasting scheme presentation.

\begin{definition}\label{def:graph}
A \emph{graph} is a tuple $G=(V_G,E_G,\psi_G)$ consisting of:
\begin{itemize}
\item a finite set $V_G$ of \emph{vertices} with at least two elements;
\item a finite set $E_G$ of \emph{edges} with at least two elements such that $E_G \cap V_G = \varnothing$;
\item an \emph{incidence function} $\psi_G : E_G\to V_G^{\times 2}$.  For each edge $e$, if $\psi_G(e)=(u,v)$, then $u$ and $v$ are called the \emph{tail} and the \emph{head} of $e$, respectively, and together they are called the \emph{ends} of $e$.  
\end{itemize} 
Moreover:
\begin{enumerate}
\item The \emph{geometric realization} of a graph $G$ is the topological quotient
\[|G| = \Bigl[\bigl(\coprodover{v\in V_G} \{v\}\bigr) \coprod \bigl(\coprodover{e\in E_G} [0,1]_e\bigr)\Bigr]\Big/\sim\]
in which:
\begin{itemize}
\item $\{v\}$ is a one-point space indexed by a vertex $v$.
\item Each $[0,1]_e$ is a copy of the topological unit interval $[0,1]$  indexed by an edge $e$.
\item The identification $\sim$ is generated by
\[u \sim 0 \in [0,1]_e \ni 1 \sim v \ifspace \psi_G(e)=(u,v).\]
\end{itemize}
\item A \emph{plane graph} is a graph together with a topological embedding of its geometric realization into the complex plane $\fieldc$.
\end{enumerate}
\end{definition}

Each vertex $v$ is drawn as a circle $\raisebox{-.05cm}{\scalebox{.8}{\begin{tikzpicture}\node [draw,circle,thick,minimum size=.4cm,inner sep=0pt] {$v$};\end{tikzpicture}}}$ with the name of the vertex inside.  Each edge $e$ with tail $u$ and head $v$ is drawn as an arrow from $u$ to $v$, as in $\scalebox{.7}{\begin{tikzpicture}
\node [smallplain] (u) {$u$};
\node [smallplain, right=.6cm of u] (v) {$v$};
\draw [arrow] (u) to node{$e$} (v); 
\end{tikzpicture}}$.
A plane graph is a graph together with a drawing of it in the complex plane $\fieldc$ such that its edges meet only at their ends.  To simplify the notation, we will identify a plane graph $G$ with its geometric realization $|G|$ and with the latter's topologically embedded image in $\fieldc$.

\begin{definition}\label{def:graph-path}
Suppose $G= (V_G,E_G,\psi_G)$ is a graph.
\begin{enumerate}
\item A \emph{path} in $G$ is an alternating sequence $v_0e_1v_1\cdots e_nv_n$ with $n\geq 0$ of vertices $v_i$'s and edges $e_i$'s such that:
\begin{itemize}
\item each $e_i$ has ends $\{v_{i-1},v_i\}$;
\item the vertices $v_i$'s are distinct.
\end{itemize}
This is also called a \emph{path from $v_0$ to $v_n$}.  A path is \emph{trivial} if $n=0$, and is \emph{non-trivial} if $n\geq 1$.
\item If $p = v_0e_1v_1\cdots e_nv_n$ is a path, then $p^* = v_ne_n\cdots v_1e_1v_0$ is the \emph{reversed path} from $v_n$ to $v_0$. 
\item A \emph{directed path} is a path such that each $e_i$ has head $v_i$.   
\item $G$ is \emph{connected} if for each pair of distinct vertices $\{u,v\}$, there exists a path from $u$ to $v$.
\end{enumerate}
\end{definition}

Using the orientation of the complex plane $\fieldc$, we identify two connected plane graphs if they are connected by a homeomorphism that preserves the orientation and the incidence relation, and that maps vertices to vertices and edges to edges.

\begin{definition}\label{def:faces}
Suppose $G$ is a connected plane graph.
\begin{enumerate}
\item The connected subspaces of the complement $\fieldc \setminus |G|$ are called the \emph{open faces} of $G$.  Their closures are called \emph{faces} of $G$.  The unique unbounded face is called the \emph{exterior face}, denoted by $\ext_G$.  The bounded faces are called  \emph{interior faces}.
\item The vertices and edges in the boundary $\bd_F$ of a face $F$ of $G$ form an alternating sequence $v_0e_1v_1\cdots e_nv_n$ of vertices and edges such that:
\begin{itemize}
\item $v_0 = v_n$.
\item The ends of $e_i$ are $\{v_{i-1},v_i\}$.
\item Traversing $\bd_F$ from $v_0$ to $v_n=v_0$ along the edges $e_1, e_2,\ldots,e_n$ in this order, ignoring their tail-to-head orientation, the face $F$ is always on the right-hand side.
\end{itemize}
\item An interior face $F$ of $G$ is \emph{anchored} if it is equipped with
\begin{itemize}
\item two distinct vertices $s_F$ and $t_F$, called the \emph{source} and the \emph{sink} of $F$, respectively, and
\item two directed paths $\dom_F$ and $\codom_F$ from $s_F$ to $t_F$, called the \emph{domain} and the \emph{codomain} of $F$, respectively,
\end{itemize}  
such that $\bd_F = \dom_F \codom_F^*$ with the first vertex in $\codom_F^*=t_F$ removed on the right-hand side. 
\item The exterior face of $G$ is \emph{anchored} if it is equipped with
\begin{itemize}
\item two distinct vertices $s_G$ and $t_G$, called the \emph{source} and the \emph{sink} of $G$, respectively, and
\item two directed paths $\dom_G$ and $\codom_G$ from $s_G$ to $t_G$, called the \emph{domain} and the \emph{codomain} of $G$, respectively,
\end{itemize}  
such that $\bd_{\ext_G} = \codom_G\dom_G^*$ with the first vertex in $\dom_G^*=t_G$ removed on the right-hand side. 
\item $G$ is \emph{anchored} if every face of $G$ is anchored.
\item $G$ is an \emph{atomic graph} if it is an anchored graph with exactly one interior face.
\end{enumerate}
\end{definition}

In an anchored graph, the boundary of each interior face is oriented clockwise.  On the other hand, the boundary of the exterior face is oriented counter-clockwise.

\begin{example}\label{ex:atomic-graph}
Here is an atomic graph $G$
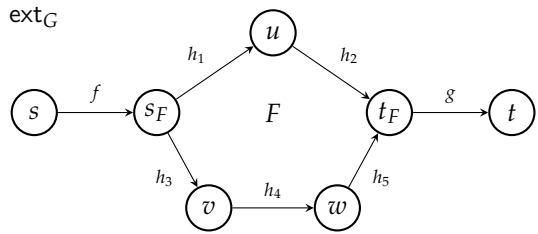
\[\begin{tikzpicture}[xscale=.7, yscale=.3]
\node [plain] (s) {$s$}; \node [plain, right=1cm of s] (s1) {$s_F$}; 
\node [right=1cm of s1] (F) {$F$}; \node [plain, above=.5cm of F] (u) {$u$}; \node [plain, below left=.8cm and .4cm of F] (v) {$v$};
\node [plain, below right=.8cm and .4cm of F] (w) {$w$}; 
\node [plain, right=1cm of F] (t1) {$t_F$};
\node [plain, right=1cm of t1] (t) {$t$};
\node [above=.7cm of s] () {$\ext_G$};
\draw [arrow] (s) to node {\scriptsize{$f$}} (s1); 
\draw [arrow] (s1) to node {\scriptsize{$h_1$}} (u);
\draw [arrow] (s1) to node[swap] {\scriptsize{$h_3$}} (v); 
\draw [arrow] (v) to node {\scriptsize{$h_4$}} (w); 
\draw [arrow] (u) to node {\scriptsize{$h_2$}} (t1);
\draw [arrow] (w) to node[swap] {\scriptsize{$h_5$}} (t1); 
\draw [arrow] (t1) to node {\scriptsize{$g$}} (t);
\end{tikzpicture}\]
with:
\begin{itemize}
\item unique interior face $F$ with source $s_F$, sink $t_F$, $\dom_F = s_F h_1 u h_2 t_F$, and $\codom_F = s_F h_3 v h_4 w h_5 t_F$;
\item exterior face $\ext_G$ with source $s$, sink $t$, $\dom_G = sfs_F h_1 u h_2 t_Fgt$, and $\codom_G =sfs_F h_3 v h_4 w h_5 t_Fgt$.\dqed
\end{itemize}
\end{example}

\begin{lemma}\label{atomic-domain}
If $G$ is an atomic graph with unique interior face $F$, then 
\[\dom_F \subseteq \dom_G \andspace \codom_F \subseteq \codom_G.\]
\end{lemma}

\begin{proof}
Since $G$ only has one interior face, the boundary $\bd_{\ext_G}=\codom_G\dom_G^*$ of the exterior face contains all of its edges.  Traversing an edge $e$ in $\dom_F$ from its tail to its head, $F$ is on the right-hand side, so $\ext_G$ is on the left-hand side.  Therefore, $e$ cannot be contained in the directed path $\codom_G$.  This proves the first containment.  The second containment is proved similarly.
\end{proof}

In particular, each atomic graph $G$ consists of its unique interior face $F$, a directed path from the source $s$ of $G$ to the source $s_F$ of $F$, and a directed path from the sink $t_F$ of $F$ to the sink $t$ of $G$.  Next we define a composition of anchored graphs that mimics the vertical composition of $2$-cells in a bicategory.

\begin{definition}\label{def:anchored-composition}
Suppose $G$ and $H$ are anchored graphs such that $s_G=s_H$, $t_G=t_H$, and $\codom_G = \dom_H$.  The \emph{vertical composite} $HG$ is the anchored graph defined by the following data.
\begin{itemize}
\item The connected plane graph of $HG$ is the quotient 
\[\dfrac{G \sqcup H}{\bigl\{\codom_G\,=\,\dom_H\bigr\}}\] of the disjoint union of $G$ and $H$, with the codomain of $G$ identified with the domain of $H$.
\item The interior faces of $HG$ are the interior faces of $G$ and $H$, which are already anchored.
\item The exterior face of $HG$ is the intersection of $\ext_G$ and $\ext_H$, with source $s_G=s_H$, sink $t_G=t_H$, domain $\dom_G$, and codomain $\codom_H$.
\end{itemize}
\end{definition}

The following observation follows from a simple inspection.

\begin{lemma}\label{graph-comp-associative}
If $G$, $H$, and $I$ are anchored graphs such that the vertical composites $IH$ and $HG$ are defined, then $(IH)G=I(HG)$.
\end{lemma}

With this lemma, we will safely omit parentheses when we write iterated vertical composites of anchored graphs.

\begin{definition}\label{def:pasting-scheme}
A \emph{pasting scheme} is an anchored graph $G$ together with a decomposition $G=G_n\cdots G_1$, called a \emph{pasting scheme presentation} of $G$, into vertical composites of $n \geq 1$ atomic graphs $G_1,\ldots,G_n$.
\end{definition}

Since the horizontal composition in a bicategory is not strictly associative, we need to equip the graphs with bracketings, which we define next.

\begin{definition}\label{def:bracketing}
\emph{Bracketings} are defined recursively as follows:
\begin{itemize}
\item The only bracketing of length $0$ is the empty sequence $\varnothing$. 
\item The only bracketing of length $1$ is the symbol $-$, called a dash.
\item If $b$ and $b'$ are bracketings of lengths $m$ and $n$, respectively, then $(bb')$ is a bracketing of length $m+n$.
\end{itemize}
We usually omit the outermost pair of parentheses, so the unique bracketing of length $2$ is $- -$.  A \emph{left normalized bracketing} is either $-$ or $(b)-$ with $b$ a left normalized bracketing.
\end{definition}

\begin{definition}\label{def:bracketing-directed-path}
For a directed path $P=v_0e_1v_1\cdots e_nv_n$ in a graph, a \emph{bracketing for $P$} is a choice of a bracketing $b$ of length $n$.  In this case, we write $b(P)$, called a \emph{bracketed directed path}, for the bracketed sequence obtained from $b$ by replacing its $n$ dashes with $e_1,\ldots,e_n$ from left to right.  If the bracketing is clear from the context, then we abbreviate $b(P)$ to $(P)$ or even $P$.  We sometimes suppress the vertices and write $P$ as $(e_1,\ldots,e_n)$, in which case $b(P)$ is also denoted by $b(e_1,\ldots,e_n)$. 
\end{definition}

\begin{example}\label{ex:bracketed-directed-path}
A directed path $P=(e_1,\ldots,e_n)$ with $0 \leq n \leq 2$ has a unique bracketing.  The only bracketings of length $3$ are $(- -)-$ and $-(- -)$.  The five bracketings of length $4$ are  $((- -)-)-$, $(- -)(- -)$, $-(-(- -))$, $(-(- -))-$, and $-((- -)-)$.  An induction shows that, for each $n\geq 1$, there is a unique left normalized bracketing of length $n$.  If $P=(e_1,e_2,e_3,e_4)$ is a directed path in a graph, then $b(P)$ for the five possible bracketings for $P$ are the bracketed sequences $((e_1e_2)e_3)e_4$, $(e_1e_2)(e_3e_4)$, $e_1(e_2(e_3e_4))$, $(e_1(e_2e_3))e_4$, and $e_1((e_2e_3)e_4)$.\dqed
\end{example}

\begin{definition}\label{def:bracketed-graph}
A \emph{bracketing} for an anchored graph $G$ consists of a bracketing $b$ for each of the directed paths $\dom_G$, $\codom_G$, $\dom_F$, and $\codom_F$ for each interior face $F$ of $G$.  An anchored graph $G$ with a bracketing is called a \emph{bracketed graph}.
\end{definition}

\begin{definition}\label{def:bracketed-graph-vcomp}
Suppose $G$ and $H$ are bracketed graphs such that:
\begin{itemize}
\item The vertical composite $HG$ of underlying anchored graphs is defined as in \Cref{def:anchored-composition}.
\item $(\codom_G) = (\dom_H)$ as bracketed directed paths.
\end{itemize} 
Then the anchored graph $HG$ is given the bracketing determined as follows:
\begin{itemize}
\item $(\dom_{HG}) = (\dom_G)$;
\item $(\codom_{HG}) = (\codom_H)$;
\item Each interior face $F$ of $HG$ is either an interior face of $G$ or an interior face of $H$, and not both.  Corresponding to these two cases, the directed paths $\dom_F$ and $\codom_F$ are bracketed as they are in $G$ or $H$.  
\end{itemize}
Equipped with this bracketing, $HG$ is called the \emph{vertical composite} of the bracketed graphs $G$ and $H$.
\end{definition}

\begin{remark}
  Note that interior faces of a bracketed graph may be bracketed incompatibly; this often arises in practice as we shall see.  Thus a bracketed graph may not decompose as a nontrivial composite, even if its underlying anchored graph does so.\dqed
\end{remark}

Vertical composition of bracketed graphs is strictly associative, so we will safely omit parentheses when we write iterated vertical composites of bracketed graphs.  Next is the graph theoretic version of a $2$-cell whiskered with a number of $1$-cells.

\begin{definition}\label{def:consistent-bracketing}
Suppose $G$ is an atomic graph with
\begin{itemize}
\item unique interior face $F$,
\item $P=(e_1,\ldots,e_m)$ the directed path from $s_G$ to $s_F$, and
\item $P'=(e'_1,\ldots,e'_n)$ the directed path from $t_F$ to $t_G$,
\end{itemize} 
as displayed below with each edge representing a directed path. 
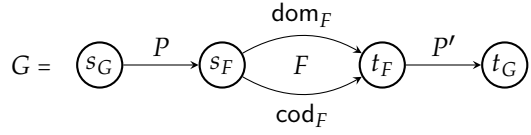
\[\begin{tikzpicture}[scale=1]
\node [plain] (s) {$s_G$}; \node [left=.2cm of s](){$G=$};
\node [plain, right=1cm of s] (s1) {$s_{F}$}; 
\node [right=.5cm of s1] () {$F$};
\node [plain, right=1.5cm of s1] (t1) {$t_{F}$};
\node [plain, right=1cm of t1] (t) {$t_{G}$};
\draw [arrow] (s) to node{$P$} (s1); 
\draw [arrow, bend left=30] (s1) to node{$\dom_{F}$} (t1); 
\draw [arrow, bend right=30] (s1) to node[swap]{$\codom_{F}$} (t1);
\draw [arrow] (t1) to node{$P'$} (t);
\end{tikzpicture}\]
A bracketing for $G$ is \emph{consistent} if it satisfies both
\begin{equation}\label{consistent-bracketing}
\begin{split}
(\dom_G) &= b\bigl(e_1,\ldots,e_m, (\dom_F), e'_1,\ldots,e'_n\bigr),\\
(\codom_G) &= b\bigl(e_1,\ldots,e_m, (\codom_F), e'_1,\ldots,e'_n\bigr)
\end{split}
\end{equation}
for some bracketing $b$ of length $m+n+1$.  In $(\dom_G)$, the bracketed directed path $(\dom_F)$ is substituted into the $(m+1)$st dash in $b$, and similarly in $(\codom_G)$.  An atomic graph with a consistent bracketing is called a \emph{consistent graph}.
\end{definition}

As we will see later, the following kind of graphs are designed for the associator and its inverse in a bicategory.

\begin{definition}\label{def:associativity-graph}
An \emph{associativity graph} is a consistent graph in which the unique interior face $F$ satisfies one of the following two conditions:
\begin{equation}\label{associativity-graph1}
(\dom_{F}) = (E_1E_2)E_3 \andspace 
(\codom_{F}) = E_1'(E_2'E_3'),
\end{equation}
or 
\begin{equation}\label{associativity-graph2}
(\dom_{F}) = E_1(E_2E_3) \andspace 
(\codom_{F}) = (E_1'E_2')E_3'.
\end{equation}
Moreover, in each case and for each $1\leq i \leq 3$, $E_i$ and $E'_i$ are non-trivial bracketed directed paths with the same length and the same bracketing.
\end{definition}

\begin{definition}\label{def:bicategorical-pasting-scheme}
A \emph{composition scheme} is a bracketed graph $G$ together with a decomposition $G=G_n\cdots G_1$, called a \emph{composition scheme presentation} of $G$, into vertical composites of $n \geq 1$ consistent graphs $G_1,\ldots,G_n$.  
\end{definition}

If $G$ is a bracketed graph that admits a composition scheme presentation $G_n\cdots G_1$, then:
\begin{itemize}
\item $G$ has $n$ interior faces, one in each consistent graph $G_i$ for $1\leq i \leq n$.
\item Each $G_i$ has the same source and the same sink as $G$.  
\item For each $1\leq i \leq n-1$, $(\codom_{G_i}) = (\dom_{G_{i+1}})$ as bracketed directed paths.
\item $(\dom_{G}) = (\dom_{G_1})$ and $(\codom_{G}) = (\codom_{G_n})$.
\item If $1\leq i\leq j\leq n$, then $G_j\cdots G_i$ is a composition scheme.
\end{itemize}

\begin{remark}[Composition schemes in 2-categories]\label{remark:2-cat-pasting-scheme}
In 2-category theory, associators are identities and therefore one typically does not distinguish between the notions of pasting scheme and composition scheme.  However the distinction is important in bicategory theory precisely because associators are typically nontrivial.
The graphs one encounters in practice often do \emph{not} admit any composition scheme presentation due to mismatched bracketings.  However, they can be extended to composition schemes in the sense of the next two definitions.\dqed
\end{remark}


\begin{definition}\label{def:collapsing}
Suppose $G$ is a bracketed graph with a decomposition as $G = G_2AG_1$, $G_2A$, or $AG_1$ into a vertical composite of bracketed graphs in which $A$ is an associativity graph with unique interior face $F$.  Using the notations in \Cref{def:associativity-graph}, the bracketed graph obtained from $G$ by identifying each edge in $E_i$ with its corresponding edge in $E'_i$ for each $1\leq i \leq 3$, along with their corresponding tails and heads, is said to be obtained from $G$ by \emph{collapsing $A$}, denoted by $G/A$.
\end{definition}

In the context of \Cref{def:collapsing}:
\begin{itemize}
\item $(\dom_{G/A}) = (\dom_G)$ and $(\codom_{G/A}) = (\codom_{G})$.
\item The interior faces in $G/A$ are those in $G$ minus the interior face of $A$, and their (co)domains are bracketed as they are in $G$.  
\item Collapsing associativity graphs is a strictly associative operation.  So we can iterate the collapsing process without worrying about the order of the collapses.  
\item If $G$ originally has the form $G_2AG_1$, then the bracketed graph $G/A$ is \emph{not} the vertical composite $G_2G_1$ of the bracketed graphs $G_1$ and $G_2$ because
\[(\codom_{G_1}) = (\dom_A) \not= (\codom_A) = (\dom_{G_2})\] as bracketed directed paths.  However, forgetting the bracketings, the underlying anchored graph of $G/A$ is the vertical composite of the underlying anchored graphs of $G_1$ and $G_2$.\
\end{itemize}

\begin{definition}\label{def:bicat-pasting-extension}
Suppose $G$ is a bracketed graph.  A \emph{composition scheme extension} of $G$ consists of the following data.
\begin{enumerate}
\item A composition scheme $H=H_n\cdots H_1$ as in \Cref{def:bicategorical-pasting-scheme}.
\item A proper subsequence of associativity graphs $\{A_1,\ldots,A_j\}$ in $\{H_1,\ldots,H_n\}$ such that $G$ is obtained from $H$ by collapsing $A_1,\ldots,A_j$.
\end{enumerate}
In this case, we also denote the bracketed graph $G$ by $H/\{A_1,\ldots,A_j\}$.
\end{definition}

In the context of \Cref{def:bicat-pasting-extension}:
\begin{itemize}
\item $(\dom_{G}) = (\dom_H)$ and $(\codom_{G}) = (\codom_{H})$.
\item The interior faces in $G$ are those in $H$ minus those in $\{A_1,\ldots,A_j\}$, and their (co)domains are bracketed as they are in $H$.
\item The order in which the associativity graphs $A_1,\ldots,A_j$ are collapsed does not matter.
\end{itemize}

To characterize bracketed graphs that admit a composition scheme extension, we need the following observation about moving brackets via associativity graphs.

\begin{lemma}\label{moving-brackets}
Suppose $G$ is a bracketed atomic graph with interior face $F$ such that:
\begin{itemize}
\item $(\dom_G) = (\dom_F)$ and $(\codom_G) = (\codom_F)$ as bracketed directed paths.
\item $(\dom_G)$ and $(\codom_G)$ have the same length.
\end{itemize} 
Then one of the following two statements holds.
\begin{enumerate}
\item $(\dom_G) = (\codom_G)$.
\item There exists a canonical vertical composite $A_k\cdots A_1$ of associativity graphs such that $(\dom_{A_1}) = (\dom_G)$ and $(\codom_{A_k}) = (\codom_G)$.
\end{enumerate}
\end{lemma}

\begin{proof}
Suppose $(\dom_G)$ and $(\codom_G)$ have length $n$, and $b_n^l$ is the left normalized bracketing of length $n$.  First we consider the case where
\[(\codom_G) = b_n^l(e_1,\ldots,e_n) = b^l_{n-1}(e_1,\ldots,e_{n-1})e_n.\]  
We proceed by induction on $n$.  If $n\leq 2$, then there is a unique bracketing of length $n$, so $(\dom_G) = b_n^l$.  

Suppose $n\geq 3$.  Then $(\dom_G)=E_1E_2$ for some canonical, non-trivial bracketed directed paths $E_1$ and $E_2$.  If $E_2$ has length $1$ (i.e., it contains the single edge $e_n$), then the induction hypothesis applies with $E_1$ as the domain and $b^l_{n-1}(e_1,\ldots,e_{n-1})$ as the codomain.  Since adding an edge at the end of an associativity graph yields an associativity graph, we are done in this case.

If $E_2$ has length $> 1$, then it has the form $E_2 = E_{21}E_{22}$ for some canonical, non-trivial bracketed directed paths $E_{21}$ and $E_{22}$.  There is a unique associativity graph $A_1$ of the form \eqref{associativity-graph2} that satisfies
\[\begin{split}
(\dom_{A_1}) &= E_1(E_{21}E_{22}) = (\dom_G),\\
(\codom_{A_1}) &= (E_1E_{21})E_{22}.\end{split}\]  
Now we repeat the previous argument with $(\codom_{A_1})$ as the new domain.  This procedure must stop after a finite number of steps because $\dom_G$ has finite length.  When it stops, the right-most bracketed directed path $E_?$ has length $1$, so we can apply the induction hypothesis as above.  This finishes the induction.

An argument dual to the above shows that $b_n^l(e_1,\ldots,e_n)$ and $(\codom_G)$ are connected by a canonical finite sequence of associativity graphs of the form \eqref{associativity-graph1}.  Splicing the two vertical composites of associativity graphs together yields the desired vertical composite.
\end{proof}

The main result of this section is the following characterization of bracketed graphs that admit a composition scheme extension.

\begin{theorem}\label{bicat-pasting-existence}
For a bracketed graph $G$, the following two statements are equivalent.
\begin{enumerate}
\item $G$ admits a composition scheme extension.
\item The underlying anchored graph of $G$ admits a pasting scheme presentation.  
\end{enumerate}
\end{theorem}

\begin{proof}
For the implication (1) $\Rightarrow$ (2), suppose $H = H_n\cdots H_1$ is a composition scheme.  By definition, this is also a pasting scheme presentation for the underlying anchored graph of $H$ because each consistent graph $H_i$ has an underlying atomic graph.  If $\{A_i\}_{1\leq i \leq j}$ is a proper subsequence of associativity graphs in $\{H_i\}_{1\leq i \leq n}$, then the vertical composite of the remaining underlying atomic graphs in \[\{H_i\}_{1\leq i \leq n} \setminus \{A_i\}_{1\leq i \leq j}\] is defined.  Moreover, it is a pasting scheme presentation for the underlying anchored graph of the bracketed graph $H/\{A_i\}_{1\leq i \leq j}$.  

For the implication (2) $\Rightarrow$ (1), suppose $G=G_m\cdots G_1$ is a pasting scheme presentation for the underlying anchored graph of $G$.  For each $1\leq i \leq m$, let $F_i$ denote the unique interior face of $G_i$, let $P_i$ denote the directed path in $G_i$ from $s_G$ to $s_{F_i}$, and let $P_i'$ denote the directed path in $G_i$ from $t_{F_i}$ to $t_G$.  Equip $G_i$ with the consistent bracketing in which:
\begin{itemize}
\item $(\dom_{F_i})$ and $(\codom_{F_i})$ are bracketed as they are in $G$;
\item $(\dom_{G_i}) = \bigl((P_i)(\dom_{F_i})\bigr)(P_i')$;
\item $(\codom_{G_i}) = \bigl((P_i)(\codom_{F_i})\bigr)(P_i')$.
\end{itemize}
Here $(P_i)$ and $(P_i')$ are either empty or left normalized bracketings.  By \Cref{moving-brackets}:
\begin{itemize}
\item Either $(\dom_G) = (\dom_{G_1})$, or else there is a vertical composite of associativity graphs $A_{1k_1}\cdots A_{11}$ with domain $(\dom_G)$ and codomain $(\dom_{G_1})$.
\item For each $2\leq  i \leq m$, either $(\codom_{G_{i-1}}) = (\dom_{G_{i}})$, or else there is a vertical composite of associativity graphs $A_{ik_i}\cdots A_{i1}$ with domain $(\codom_{G_{i-1}})$ and codomain $(\dom_{G_{i}})$.
\item Either $(\codom_{G_m}) = (\codom_G)$, or else there is a vertical composite of associativity graphs $A_{m+1,k_{m+1}}\cdots A_{m+1,1}$ with domain $(\codom_{G_m})$ and codomain $(\codom_G)$.
\end{itemize} 
The corresponding vertical composite
\[H = \overbracket[0.5pt]{(A_{m+1,k_{m+1}}\cdots A_{m+1,1})}^{\text{or $\varnothing$}} G_m \cdots \overbracket[0.5pt]{(A_{2k_2}\cdots A_{21})}^{\text{or $\varnothing$}} G_1 \overbracket[0.5pt]{(A_{1k_1}\cdots A_{11})}^{\text{or $\varnothing$}}\]
is a composition scheme.  Moreover, $G$ is obtained from $H$ by collapsing all the associativity graphs $A_{ij}$ for $1\leq i \leq m+1$ and $1 \leq j \leq k_i$.
\end{proof}


\section{Pasting Diagrams and Composition Diagrams}\label{sec:pasting-diagrams}

In this section we apply the graph theoretic concepts in the previous
section to define pasting diagrams and composition diagrams in
bicategories.
We begin with the definition of a bicategory.  In what follows, $\boldone$ denotes the discrete category with one object $*$.  For a category $\C$, we identify the categories $\C\times \boldone$ and $\boldone \times \C$ with $\C$, and regard the canonical isomorphisms between them as $\Id_{\C}$.

\begin{definition}\label{def:bicategory}
A \emph{bicategory} is a tuple $\bigl(\B, 1, c, a, \ell, r\bigr)$ consisting of the following data.
\begin{enumerate}[label=(\roman*)]
\item $\B$ is equipped with a collection $\Ob(\B) = \B_0$, whose elements are called \emph{objects} in $\B$.  If $X\in \B_0$, we also write $X\in \B$.
\item For each pair of objects $X,Y\in\B$, $\B$ is equipped with a category $\B(X,Y)$, called a \emph{hom category}.
\begin{itemize}
\item Its objects are called \emph{$1$-cells}, and its morphisms are called \emph{$2$-cells} in $\B$.
\item Composition and identity morphisms in $\B(X,Y)$ are called \emph{vertical composition} and \emph{identity $2$-cells}, respectively.
\item For a $1$-cell $f$, its identity $2$-cell is denoted by $1_f$.
\end{itemize}
\item For each object $X\in\B$, $1_X : \boldone \to \B(X,X)$ is a functor, which we identify with the $1$-cell $1_X(*)\in\B(X,X)$, called the \emph{identity $1$-cell of $X$}.
\item For each triple of objects $X,Y,Z \in \B$, 
\[c_{XYZ} : \B(Y,Z) \times \B(X,Y) \to \B(X,Z)\]
is a functor, called the \emph{horizontal composition}.  For $1$-cells $f \in \B(X,Y)$ and $g \in \B(Y,Z)$, and $2$-cells $\alpha \in \B(X,Y)$ and $\beta \in \B(Y,Z)$, we use the notations
\[c_{XYZ}(g,f) = gf \andspace c_{XYZ}(\beta,\alpha) = \beta * \alpha.\]
\item For objects $W,X,Y,Z \in \B$, 
\[a_{WXYZ} : c_{WXZ} \bigl(c_{XYZ} \times \Id_{\B(W,X)}\bigr) \to c_{WYZ}\bigl(\Id_{\B(Y,Z)} \times c_{WXY}\bigr)\]
is a natural isomorphism, called the \emph{associator}.
\item For each pair of objects $X,Y\in \B$,
\[\begin{tikzcd}
c_{XYY} \bigl(1_Y \times \Id_{\B(X,Y)}\bigr) \rar{\ell_{XY}} & \Id_{\B(X,Y)}
& c_{XXY}\bigl(\Id_{\B(X,Y)} \times 1_X\bigr) \lar[swap]{r_{XY}}\end{tikzcd}\]
are natural isomorphisms, called the \emph{left unitor} and the \emph{right unitor}, respectively.
\end{enumerate}
The subscripts in $c$ will often be omitted.  The subscripts in $a$, $\ell$, and $r$ will often be used to denote their components.  The above data is required to satisfy the following two axioms for $1$-cells $f \in \B(V,W)$, $g \in \B(W,X)$, $h \in \B(X,Y)$, and $k \in \B(Y,Z)$.
\begin{description}
\item[Unity Axiom] The middle unity diagram
\[\begin{tikzcd}[column sep=small] (g 1_W)f \arrow{rr}{a} \arrow{rd}[swap]{r_g * 1_f} && g(1_W f) \arrow{ld}{1_g * \ell_f}\\ & gf
\end{tikzcd}\]
in $\B(V,X)$ is commutative.
\item[Pentagon Axiom]
The diagram
\begin{equation}\label{bicat-pentagon}
\begin{tikzpicture}[commutative diagrams/every diagram]
\node (P0) at (90:2cm) {$(kh)(gf)$};
\node (P1) at (90+72:2cm) {$((kh)g)f$} ;
\node (P2) at (220:1.5cm) {\makebox[3ex][r]{$(k(hg))f$}};
\node (P3) at (-40:1.5cm) {\makebox[3ex][l]{$k((hg)f)$}};
\node (P4) at (90+4*72:2cm) {$k(h(gf))$};
\path[commutative diagrams/.cd, every arrow, every label]
(P0) edge node {$a_{k,h,gf}$} (P4)
(P1) edge node {$a_{kh,g,f}$} (P0)
(P1) edge node[swap] {$a_{k,h,g} * 1_f$} (P2)
(P2) edge node {$a_{k,hg,f}$} (P3)
(P3) edge node[swap] {$1_k * a_{h,g,f}$} (P4);
\end{tikzpicture}
\end{equation}
in $\B(V,Z)$ is commutative.
\end{description}
This finishes the definition of a bicategory.
\end{definition}

\begin{remark}\label{expl:bicategory}
Suppose $\B$ is a bicategory.
\begin{itemize}
\item We assume the hom categories $\B(X,Y)$ for objects $X,Y\in\B$ are disjoint.  If not, we tacitly replace them with their disjoint union.
\item For $2$-cells $\alpha : f \to f'$, $\alpha' : f' \to f''$, and $\alpha'' : f'' \to f'''$ in $\B(X,Y)$, there are equalities
\begin{equation}\label{hom-category-axioms}
(\alpha'' \alpha') \alpha = \alpha'' (\alpha' \alpha) \andspace
\alpha= \alpha 1_f = 1_{f'} \alpha.
\end{equation}
\item With the usual notation
\[\begin{tikzpicture}[commutative diagrams/every diagram]
\node (X) at (-1,0) {$X$}; \node (Y) at (1,0) {$Y$};
\node[font=\Large] at (-.1,0) {\rotatebox{270}{$\Rightarrow$}}; 
\node at (.15,0) {$\alpha$};
\path[commutative diagrams/.cd, every arrow, every label] 
(X) edge [bend left] node[above] {$f$} (Y)
edge [bend right] node[below] {$f'$} (Y);
\end{tikzpicture}\]
for a $2$-cell $\alpha : f \to f'$, the horizontal composition $c_{XYZ}$ is the assignment:
\[\begin{tikzpicture}[commutative diagrams/every diagram]
\node (X) at (-1,0) {$X$}; \node (Y) at (1,0) {$Y$}; \node (Z) at (3,0) {$Z$};
\node (X1) at (5,0) {$X$}; \node (Z1) at (7,0) {$Z$};
\node[font=\Large] at (-.1,0) {\rotatebox{270}{$\Rightarrow$}}; 
\node at (.15,0) {$\alpha$};
\node[font=\Large] at (1.9,0) {\rotatebox{270}{$\Rightarrow$}}; 
\node at (2.15,0) {$\beta$};
\node[font=\Large] at (5.7,0) {\rotatebox{270}{$\Rightarrow$}}; 
\node at (6.2,0) {$\beta*\alpha$};
\node at (4,0) {$\mapsto$};
\path[commutative diagrams/.cd, every arrow, every label] 
(X) edge [bend left] node[above] {$f$} (Y)
(X) edge [bend right] node[below] {$f'$} (Y)
(Y) edge [bend left] node[above] {$g$} (Z)
(Y) edge [bend right] node[below] {$g'$} (Z)
(X1) edge [bend left] node[above] {$gf$} (Z1)
(X1) edge [bend right] node[below] {$g'f'$} (Z1);
\end{tikzpicture}\]
\item There are equalities
\begin{equation}\label{bicat-c-id}
1_g * 1_f = 1_{gf}
\end{equation}
in $\B(X,Z)(gf,gf)$, and
\begin{equation}\label{middle-four}
(\beta'\beta) * (\alpha'\alpha) = (\beta'*\alpha')(\beta*\alpha)
\end{equation}
in $\B(X,Z)(gf,g''f'')$ for $1$-cells $f'' \in \B(X,Y)$, $g'' \in \B(Y,Z)$ and $2$-cells $\alpha' : f'\to f''$, $\beta' : g' \to g''$.\dqed
\end{itemize}
\end{remark}

We now apply the graph theoretic concepts above to bicategories.

\begin{definition}\label{def:bicat-pasting-diagram}
Suppose $\B$ is a bicategory, and $G$ is a bracketed graph.  
\begin{enumerate}
\item A \emph{$1$-skeletal $G$-diagram} in $\B$ is an assignment $\phi$ as follows.
\begin{itemize}
\item $\phi$ assigns to each vertex $v$ in $G$ an object $\phi_v$ in $\B$.
\item $\phi$ assigns to each edge $e$ in $G$ with tail $u$ and head $v$ a $1$-cell $\phi_e \in \B(\phi_u,\phi_v)$. 
\end{itemize}
\item Suppose $\phi$ is such a $1$-skeletal $G$-diagram, and $P = v_0e_1v_1\cdots e_mv_m$ is a directed path in $G$ with $m \geq 1$ and with an inherited bracketing $(P)$.  Define the $1$-cell 
\begin{equation}\label{phi-directed-path}
\phi_P \in \B(\phi_{v_0},\phi_{v_m})
\end{equation}
as follows.
\begin{itemize}
\item First replace the edge $e_i$ in $(P)$ by the $1$-cell $\phi_{e_i} \in \B(\phi_{v_{i-1}},\phi_{v_i})$ for $1\leq i \leq m$.
\item Then form the horizontal composite of the resulting parenthesized sequence $\begin{tikzcd}\phi_{v_0} \ar{r}{\phi_{e_1}} & \phi_{v_1} \ar{r}{\phi_{e_2}} & \cdots \ar{r}{\phi_{e_m}} & \phi_{v_m}\end{tikzcd}$
of $1$-cells.
\end{itemize}
\item A \emph{$G$-diagram} in $\B$ is a $1$-skeletal $G$-diagram $\phi$ in $\B$ that assigns to each interior face $F$ of $G$ a $2$-cell $\phi_F : \phi_{\dom_F} \to \phi_{\codom_F}$ in $\B(\phi_{s_F},\phi_{t_F})$.
\item A $G$-diagram is called a \emph{composition diagram} of shape $G$ if $G$ admits a composition scheme presentation.

\item A $G$-diagram is called a \emph{pasting diagram} if the
  underlying anchored graph admits a pasting scheme presentation.
  Equivalently, by \cref{bicat-pasting-existence}, a $G$-diagram
  $\phi$ is a pasting diagram in $\B$ if and only if $G$ admits a
  composition scheme extension.
\end{enumerate}
\end{definition}
\begin{remark}[Pasting diagrams in 2-categories]\label{remark:2cat-pasting-diagram}
  Suppose $\B$ is a 2-category, regarded as a bicategory, and let
  $\B'$ denote its underlying 2-category. If $G$ is a bracketed graph
  and $\phi$ a $G$-diagram in $\B$, let $G'$ denote the underlying
  anchored graph of $G$ and let $\phi'$ denote the corresponding
  $G'$-diagram in $\B'$.  Then $\phi$ is a pasting diagram in $\B$ if
  and only if $\phi'$ is a pasting diagram in $\B'$.\dqed
\end{remark}

\begin{definition}\label{def:bicat-pasting-composite}
Suppose $\phi$ is a composition diagram of shape $G$ in a bicategory $\B$ and suppose $G_n\cdots G_1$ is a composition scheme presentation of $G$.
\begin{enumerate}
\item For each $1\leq i \leq n$, the \emph{constituent} 2-cell for $G_i$, denoted by $\phi_{G_i}$, is defined as follows.  Suppose $G_i$ has:
\begin{itemize}
\item unique interior face $F_i$;
\item directed path $P_i = (e_{i1},\ldots,e_{ik_i})$ from $s_G$ to $s_{F_i}$;
\item directed path $P_i' = (e'_{i1},\ldots,e'_{il_i})$ from $t_{F_i}$ to $t_G$.
\end{itemize}  
By \eqref{consistent-bracketing} the bracketing of the consistent graph $G_i$ satisfies
\[\begin{split}
(\dom_{G_i}) &= b_i\bigl(e_{i1},\ldots,e_{ik_i}, (\dom_{F_i}), e'_{i1},\ldots,e'_{il_i}\bigr),\\
(\codom_{G_i}) &= b_i\bigl(e_{i1},\ldots,e_{ik_i}, (\codom_{F_i}), e'_{i1},\ldots,e'_{il_i}\bigr)
\end{split}\]
for some bracketing $b_i$ of length $k_i+l_i+1$.  Then we define the $2$-cell 
\begin{equation}\label{basic-2cell}
\phi_{G_i} = b_i\bigl(1_{\phi_{e_{i1}}},\ldots, 1_{\phi_{e_{ik_i}}}, \phi_{F_i}, 1_{\phi_{e'_{i1}}},\ldots, 1_{\phi_{e'_{il_i}}}\bigr) : \phi_{\dom_{G_i}} \to \phi_{\codom_{G_i}}
\end{equation} 
in $\B(\phi_{s_G},\phi_{t_G})$ where:
\begin{itemize}
\item The identity $2$-cell of each $\phi_{e_{ij}}$ is substituted for $e_{ij}$ in $b_i$, and similarly for the identity $2$-cell of each $\phi_{e'_{ij}}$.  
\item The $2$-cell $\phi_{F_i}$ is substituted for the $(k_i+1)$st entry in $b_i$.
\item $\phi_{G_i}$ is the iterated horizontal composite of the resulting bracketed sequence of $2$-cells, with the horizontal compositions determined by the brackets in $b_i$.
\end{itemize}
\item The \emph{composite of $\phi$} with respect to $G_n \cdots G_1$, denoted by $|\phi|$, is defined as the vertical composite
\begin{equation}\label{pasting-diagram-composite}
\begin{tikzcd}[column sep=huge]
\phi_{\dom_G} = \phi_{\dom_{G_1}} \ar{r}{|\phi| \,=\, \phi_{G_n}\cdots \phi_{G_1}} & \phi_{\codom_{G_n}}=\phi_{\codom_G},\end{tikzcd}
\end{equation} 
which is a $2$-cell in $\B(\phi_{s_G},\phi_{t_G})$.
\end{enumerate}
\end{definition}

\begin{example}\label{ex:bicat-pasting-2}
Suppose given a $G$-diagram $\phi$ in $\B$, as displayed on the left below.  The underlying anchored graph $G$ has a unique bracketing because, in all three interior faces and the exterior face, the domain and the codomain have at most two edges.  The bracketed graph $G$ does \emph{not} admit a composition scheme presentation.
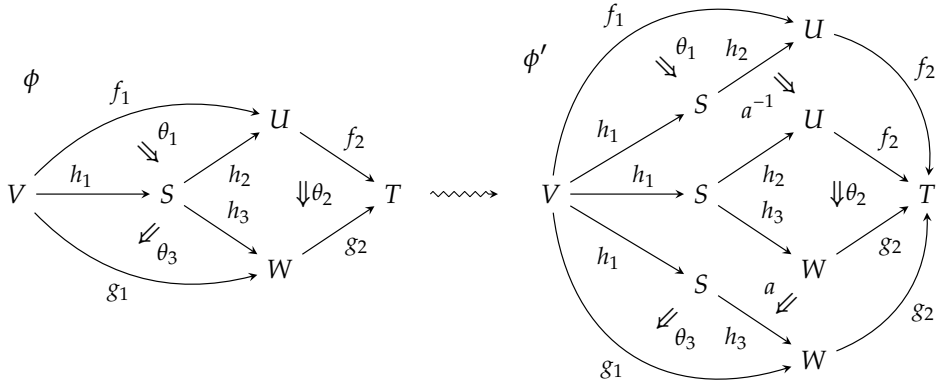
\[
\begin{tikzpicture}[x=20mm,y=20mm]
  \newcommand{\core}{
    \draw[0cell] 
    (0,0) node (v) {V}
    (1,0) node (s) {S}
    (1.75,.5) node (u) {U}
    (1.75,-.5) node (w) {W}
    (2.5,0) node (t) {T}
    ;
    \draw[1cell] 
    (s) edge[swap] node {h_2} (u)
    (s) edge node {h_3} (w)
    (u) edge node {f_2} (t)
    (w) edge[swap] node {g_2} (t)
    ;
    \draw[2cell] 
    node[between=s and t at .6, rotate=-90,font=\Large] (T2) {\Rightarrow}
    (T2) node[right] {\theta_2}  
    ;
  }
  \begin{scope}
    \core
    \draw [->, line join=round, decorate, decoration={zigzag, segment
        length=4, amplitude=1, post=lineto, post length=2pt}]
    (t) ++(.25,0) -- ++(0.45,0);
    \draw[0cell]
    (v) ++(.1,.75) node {\phi}
    ;
    \draw[1cell] 
    (v) edge node[pos=.4] {h_1} (s)
    (v) edge[bend left] node (f1) {f_1} (u)
    (v) edge[bend right, swap] node (g1) {g_1} (w)
    ;
    \draw[2cell] 
    node[between=f1 and s at .6, rotate=-45,font=\Large] (T1) {\Rightarrow}
    (T1) node[above right] {\theta_1}
    node[between=g1 and s at .6, rotate=225,font=\Large] (T3) {\Rightarrow}
    (T3) node[below right] {\theta_3}
    ;
  \end{scope}
  \begin{scope}[shift={(3.55,0)}]
    \core
    \draw[0cell] 
    (v) ++(-.1,.9) node {\phi'}
    (s) ++(0,.6) node (s') {S}
    (u) ++(0,.6) node (u') {U}
    (s) ++(0,-.6) node (s'') {S}
    (w) ++(0,-.6) node (w'') {W}
    ;
    \draw[1cell] 
    (v) edge node[pos=.65] {h_1} (s)
    (v) edge node[pos=.5] {h_1} (s')
    (v) edge[swap] node[pos=.5] {h_1} (s'')
    (s') edge node {h_2} (u')
    (s'') edge[swap] node {h_3} (w'')
    (u') edge[bend left=40] node {f_2} (t)
    (w'') edge[swap, bend right=35] node {g_2} (t)
    (v) edge[bend left=50, looseness=1.2] node (f1') {f_1} (u')
    (v) edge[bend right=50, looseness=1.2, swap] node (g1') {g_1} (w'')
    ;
    \draw[2cell] 
    node[between=f1' and s' at .6, rotate=-45,font=\Large] (T1) {\Rightarrow}
    (T1) node[above right] {\theta_1}
    node[between=g1' and s'' at .6, rotate=225,font=\Large] (T3) {\Rightarrow}
    (T3) node[below right] {\theta_3}
    (u) ++(130:.3) node[rotate=-45,font=\Large] (a') {\Rightarrow}
    (a') node[below left] {\,a^\inv}
    (w) ++(230:.3) node[rotate=225,font=\Large] (a) {\Rightarrow}
    (a) node[above left] {a}
    ;
  \end{scope}
\end{tikzpicture}
\]
The composite of $\phi$ is not defined in general because
\[\begin{split}
\codom(1_{f_2}*\theta_1) = f_2(h_2h_1) & \not= (f_2h_2)h_1 = \dom(\theta_2*1_{h_1}),\\
\codom(\theta_2*1_{h_1}) = (g_2h_3)h_1 &\not= g_2(h_3h_1) = \dom(1_{g_2}*\theta_3).\end{split}\]
We can fix the mismatched bracketings by:
\begin{itemize}
\item expanding $G$ into a composition scheme $G'$ by inserting two associativity graphs, one of the form \eqref{associativity-graph1} and the other \eqref{associativity-graph2}; 
\item inserting instances of the associator $a$ or its inverse $a^{-1}$ to obtain the composition diagram $\phi'$ of shape $G'$ on the right above.  
\end{itemize}
The composite of $\phi$ may now be defined as the vertical composite
\[\begin{tikzcd}
f_2f_1 \ar{rrr}{|\phi'|} \ar{d}[swap]{1_{f_2}*\theta_1} &&& g_2g_1\\
f_2(h_2h_1) \ar{r}{a^{-1}} & (f_2h_2)h_1 \ar{r}{\theta_2*1_{h_1}} & (g_2h_3)h_1 \ar{r}{a} & g_2(h_3h_1) \ar{u}[swap]{1_{g_2}*\theta_3} \end{tikzcd}\]
of $2$-cells in $\B(V,T)$.
\dqed
\end{example}

The essential idea demonstrated in \cref{ex:bicat-pasting-2} works in general to extend a pasting diagram to a composition diagram.  We explain this in the following two definitions.

\begin{definition}\label{def:associator-or-inverse}
Suppose $\phi$ is a $1$-skeletal $A$-diagram in a bicategory $\B$ for some associativity graph $A$.
\begin{enumerate}
\item We call $\phi$ \emph{extendable} if, using the notations in \Cref{def:associativity-graph}, for each $1\leq i \leq 3$ and each edge $e$ in $E_i$ with corresponding edge $e'$ in $E'_i$, there is an equality of $1$-cells $\phi_e = \phi_{e'}$.  As defined in \eqref{phi-directed-path}, this implies the equality $\phi_{E_i} = \phi_{E'_i}$ of composite $1$-cells.  
\item Suppose $\phi$ is extendable.  The \emph{canonical extension of $\phi$} is the $A$-diagram that assigns to the unique interior face $F$ of $A$ the $2$-cell
\[\begin{tikzcd}[column sep=huge]
\phi_{\dom_F} = \phi_{E_3}(\phi_{E_2}\phi_{E_1}) \ar{r}{\phi_F \,=\, a^{-1}} & (\phi_{E'_3}\phi_{E'_2})\phi_{E'_1} = \phi_{\codom_F}\end{tikzcd}\]
if $A$ satisfies \eqref{associativity-graph1}, or
\[\begin{tikzcd}[column sep=huge]
\phi_{\dom_F} = (\phi_{E_3}\phi_{E_2})\phi_{E_1} \ar{r}{\phi_F \,=\, a} & \phi_{E'_3}(\phi_{E'_2}\phi_{E'_1}) = \phi_{\codom_F}\end{tikzcd}\]
if $A$ satisfies \eqref{associativity-graph2}.
\end{enumerate}
\end{definition}

\begin{example}
In \Cref{ex:bicat-pasting-2} the composition diagram $\phi'$ involves two canonical extensions of restrictions of $\phi$, one for each of $a$ and $a^{-1}$.\dqed 
\end{example}

\begin{definition}\label{def:bicat-diagram-composite}
  Suppose that $\phi$ is a pasting diagram of shape $G$ in a bicategory $\B$, and suppose $H=H_n\cdots H_1$ is a composition scheme extension of $G$.
  The \emph{composite} of $\phi$ with respect to $H=H_n\cdots H_1$, denoted by $|\phi|$, is defined as follows.
\begin{enumerate}
\item First define the composition diagram $\phi_H$ of shape $H$ by the following data:
\begin{itemize}
\item The restriction of $\phi_H$ to $(\dom_H)$ is $(\dom_G)$; to $(\codom_H)$ is $(\codom_G)$; and to the interior faces in $G$, agrees with $\phi$.
\item For each $1\leq i \leq j$, the restriction of $\phi_H$ to the associativity graph $A_i$ is extendable.  The value of $\phi_H$ at the unique interior face of $A_i$ is given by the canonical extension described in \Cref{def:associator-or-inverse}(2).  That is, it is either a component of the associator $a$ or its inverse.
\end{itemize}
\item Now we define the $2$-cell $|\phi|$ in $\B(\phi_{s_G},\phi_{t_G})$ by  
\[\begin{tikzcd}[column sep=large]
\phi_{\dom_G} \ar{r}{|\phi| \,=\, |\phi_H|} & \phi_{\codom_G},\end{tikzcd}\] where $|\phi_H|$ is the composite of $\phi_H$ as in \eqref{pasting-diagram-composite} with respect to $H_n\cdots H_1$.
\end{enumerate}
\end{definition}


\section{Bicategorical Pasting Theorem}
\label{sec:bicat-pasting-theorem}

In this section we prove the Bicategorical Pasting Theorem \ref{thm:bicat-pasting-theorem}.
Existence of a composite follows from \cref{bicat-pasting-existence}
and \cref{def:bicat-diagram-composite}.
The
majority of the remaining work is to show, for a pasting diagram $\phi$ of shape $G$ in a bicategory, the composites with respect to any two composition scheme extensions of $G$ are equal.  The proof of this result restricts to $2$-categories and yields essentially Power's pasting theorem for $2$-categories.

We begin with an adaptation of Mac Lane's Coherence Theorem to this context.
\begin{theorem}[Mac Lane's Coherence]\label{maclane-coherence}
Suppose:
\begin{enumerate}
\item $G=A_k\cdots A_1$ and $G'=A'_l\cdots A'_1$ are composition schemes such that:
\begin{itemize}
\item All the $A_i$ and $A_j'$ are associativity graphs.
\item $(\dom_G) = (\dom_{G'})$ and $(\codom_G) = (\codom_{G'})$ as bracketed directed paths.
\end{itemize} 
\item $\phi$ is a $1$-skeletal $G$-diagram in $\B$ whose restriction to each $A_i$ is extendable.  With the canonical extension of $\phi$ in each $A_i$, the resulting composition diagram of shape $G$ is denoted by $\phibar$.
\item $\phi'$ is a $1$-skeletal $G'$-diagram in $\B$ whose restriction to each $A'_j$ is extendable.  With the canonical extension of $\phi'$ in each $A'_j$, the resulting composition diagram of shape $G'$ is denoted by $\phibar'$.
\item $\phi_{e} = \phi'_{e}$ for each edge $e$ in $\dom_G$.
\end{enumerate}
Then there is an equality $|\phibar| = |\phibar'|$ of composite $2$-cells in $\B(\phi_{s_G},\phi_{t_G})$.
\end{theorem}  

\begin{proof}
The desired equality is
\begin{equation}\label{desired-equality}
  \phibar_{A_k}\cdots\phibar_{A_1} = \phibar'_{A'_l}\cdots\phibar'_{A'_1}
\end{equation}
with
\begin{itemize}
\item each side a vertical composite as in \eqref{pasting-diagram-composite}, and
\item $\phibar_{A_i}$ and $\phibar'_{A'_j}$ horizontal composites as in \eqref{basic-2cell}.
\end{itemize}  
The proof that these are equal is adapted as follows from the proof of Mac Lane's Coherence Theorem for monoidal categories in \cite{maclane} (p.166-168), which characterizes the free monoidal category on one object.
\begin{itemize}
\item Suppose the edges in $\dom_G$, and hence also in $\codom_G$, are $e_1,\ldots,e_n$ from the source $s_G$ to the sink $t_G$.  By hypothesis there are equalities of $1$-cells:
\begin{itemize}
\item $\phi_{e_i} = \phi'_{e_i}$ for $1\leq i \leq n$;
\item $\phi_{\dom_G} = \phi'_{\dom_{G'}}$ and $\phi_{\codom_G} = \phi'_{\codom_{G'}}$.
\end{itemize}  
Mac Lane considered $\otimes$-words involving $n$ objects in a monoidal category.  Here we consider bracketings of the sequence of $1$-cells $(\phi_{e_1},\ldots,\phi_{e_n})$.
\item Identity morphisms within $\otimes$-words are replaced by identity $2$-cells in the ambient bicategory $\B$.
\item Each instance of the associativity isomorphism $\alpha$ in a monoidal category is replaced by a component of the associator $a$.
\item A basic arrow in Mac Lane's sense is a $\otimes$-word of length $n$ involving one instance of $\alpha$ and $n-1$ identity morphisms.  Basic arrows are replaced by $2$-cells of the forms $\phibar_{A}$ or $\phibar'_{A}$ for an associativity graph $A$.
\item Composites of basic arrows are replaced by vertical composites of $2$-cells.
\item The bifunctoriality of the monoidal product is replaced by the functoriality of the horizontal composition in $\B$.
\item The Pentagon Axiom in a monoidal category is replaced by the Pentagon Axiom \eqref{bicat-pentagon} in the bicategory $\B$.
\end{itemize}
Mac Lane's proof shows that, given any two $\otimes$-words $u$ and $w$ of length $n$ involving the same sequence of objects, any two composites of basic arrows from $u$ to $w$ are equal. With the adaptation detailed above, Mac Lane's argument yields the desired equality \eqref{desired-equality}.
\end{proof}

Now we come to our main result, the Bicategorical Pasting Theorem.
\begin{theorem}[Bicategorical Pasting]\label{thm:bicat-pasting-theorem}
  Suppose $\B$ is a bicategory.  Every pasting diagram in $\B$ has a
  unique composite.
\end{theorem}
\begin{proof}
  Suppose $G$ is a bracketed graph whose underlying anchored graph
  admits a pasting scheme extension and suppose $\phi$ is a pasting
  diagram of shape $G$ in $\B$. Existence of a composite follows from
  \cref{bicat-pasting-existence}: $G$ has a composition scheme
  extension $H$, and $\phi$ has a composite with respect to $H$ as
  described in \cref{def:bicat-diagram-composite}.

Now we turn to uniqueness.
Suppose we are given two composition scheme extensions of $G$, say 
\begin{itemize}
\item $H= H_{j+n}\cdots H_1$ with associativity graphs $\{A_1,\ldots,A_j\}$ and
\item $H' = H'_{k+n}\cdots H'_1$ with associativity graphs $\{A'_1,\ldots,A'_k\}$.
\end{itemize} 
We want to show that the composites of $\phi$ with respect to $H= H_{j+n}\cdots H_1$ and $H' = H'_{k+n}\cdots H'_1$ are the same.  The proof is an induction on the number $n$ of interior faces of $G$.  

The case $n=1$ follows from
\begin{enumerate}[label=(\roman*)]
\item \Cref{moving-brackets}, 
\item Mac Lane's Coherence \Cref{maclane-coherence}, and 
\item the naturality of the associator $a$ and its inverse 
\end{enumerate} 
as follows.  Suppose the unique interior face $F$ of $G$ appears in $H_p$ and $H'_q$ for some $1\leq p \leq j+1$ and $1\leq q \leq k+1$.  Since $H_p$ and $H_q'$ are consistent graphs, by \eqref{consistent-bracketing} there exist bracketings $b$ and $b'$ of the same length, say $m$, such that
\[\begin{split}
(\dom_{H_p}) &= b\bigl(e_1,\ldots,e_{l-1}, (\dom_F), e_{l+1},\ldots,e_m\bigr),\\
(\codom_{H_p}) &= b\bigl(e_1,\ldots,e_{l-1}, (\codom_F), e_{l+1},\ldots,e_m\bigr),\\
(\dom_{H'_q}) &= b'\bigl(e_1,\ldots,e_{l-1}, (\dom_F), e_{l+1},\ldots,e_m\bigr),\\
(\codom_{H'_q}) &= b'\bigl(e_1,\ldots,e_{l-1}, (\codom_F), e_{l+1},\ldots,e_m\bigr).
\end{split}\]
There is a unique bracketed atomic graph $C$ with interior face $C_F$ such that
\begin{itemize}
\item $(\dom_C) = (\dom_{C_F}) = (\dom_{H_p})$ and
\item $(\codom_C) = (\codom_{C_F}) = (\dom_{H'_q})$.
\end{itemize}
By (i) there exists a canonical vertical composite $C'=C_r \cdots C_1$ of associativity graphs $C_1,\ldots,C_r$ such that:
\begin{itemize}
\item $(\dom_{C'}) = (\dom_{C_1}) = (\dom_C)$.
\item $(\codom_{C'}) = (\codom_{C_r}) = (\codom_C)$.
\item No $C_i$ changes the bracketing of $(\dom_F)$. 
\end{itemize} 
Indeed, since the bracketed directed path $(\dom_F)$ appears as the $l$th entry in both $b$ and $b'$, we can first regard $(\dom_F)$ as a single edge, say $e_l$, in $C$.  Applying (i) in that setting gives a vertical composite of associativity graphs with domain $b(e_1,\ldots,e_m)$ and codomain $b'(e_1,\ldots,e_m)$.  Then we substitute $(\dom_F)$ in for each $e_l$ in the resulting vertical composite.

The sequence of edges \[\{e_1,\ldots,e_{l-1},\dom_F,e_{l+1},\ldots,e_m\}\] in $\dom_{H_p}$ is the same as those in $\dom_G$ and $\dom_{H'_q}$.  So the underlying $1$-skeletal $G$-diagram of $\phi$ uniquely determines a composition diagram $\phi_{C'}$ of shape $C'$, in which every interior face is assigned either a component of the associator $a$ or its inverse, corresponding to the two cases \eqref{associativity-graph2} and \eqref{associativity-graph1}.  Its composite with respect to the composition scheme presentation $C_r\cdots C_1$ is denoted by $|\phi_{C'}|$.  Similar remarks apply with $\codom_F$, $\codom_{H_p}$, $\codom_G$, and $\codom_{H'_q}$ replacing $\dom_F$, $\dom_{H_p}$, $\dom_G$, and $\dom_{H'_q}$, respectively.

Moreover, since $n=1$, by the definitions of $H_p$ and $H_q'$ there are equalities
\[\begin{split}
\{H_1,\ldots,H_{j+1}\} &= \{A_1,\ldots,A_{p-1},H_p,A_p,\ldots,A_j\},\\
\{H'_1,\ldots,H'_{k+1}\} &= \{A'_1,\ldots,A'_{q-1},H'_q,A'_q,\ldots,A'_k\}.
\end{split}\]
Consider the following diagram in $\B(\phi_{s_G},\phi_{t_G})$.
\[\begin{tikzpicture}[commutative diagrams/every diagram, scale=1.5]
\tikzset{arrow/.style={commutative diagrams/.cd,every arrow}}
\node (dH) at (-1,3) {$\phi_{\dom_H}$};
\node (dHprime) at (1,3) {$\phi_{\dom_{H'}}$}; 
\node (dHp) at (-1,2) {$\phi_{\dom_{H_p}}$}; 
\node (dHq) at (1,2) {$\phi_{\dom_{H_q'}}$}; 
\node (cHp) at (-1,1) {$\phi_{\codom_{H_p}}$}; 
\node (cHq) at (1,1) {$\phi_{\codom_{H'_q}}$}; 
\node (cH) at (-1,0) {$\phi_{\codom_H}$};
\node (cHprime) at (1,0) {$\phi_{\codom_{H'}}$};
\draw [arrow] (dH) to node{\scriptsize{$1_{\phi_{\dom_G}}$}} (dHprime);
\draw [arrow] (dH) to node[swap]{\scriptsize{$\phi_{A_{p-1}}\cdots\phi_{A_1}$}} (dHp); 
\draw [arrow] (dHprime) to node{\scriptsize{$\phi_{A'_{q-1}}\cdots \phi_{A'_1}$}} (dHq);
\draw [arrow] (dHp) to node{\scriptsize{$|\phi_{C'}|$}} (dHq);
\draw [arrow] (dHp) to node[swap]{\scriptsize{$\phi_{H_p}$}} (cHp);
\draw [arrow] (dHq) to node{\scriptsize{$\phi_{H'_q}$}} (cHq);
\draw [arrow] (cHp) to node{\scriptsize{$|\phi_{C'}|$}} (cHq);
\draw [arrow] (cHp) to node[swap]{\scriptsize{$\phi_{A_j}\cdots\phi_{A_p}$}} (cH); 
\draw [arrow] (cHq) to node{\scriptsize{$\phi_{A'_k}\cdots \phi_{A'_q}$}} (cHprime);
\draw [arrow] (cH) to node{\scriptsize{$1_{\phi_{\codom_G}}$}} (cHprime);
\end{tikzpicture}\]
The left-bottom boundary and the top-right boundary are the composites of $\phi$ with respect to $H= H_{j+1}\cdots H_1$ and $H' = H'_{k+1}\cdots H'_1$, respectively.  The top and bottom rectangles are commutative by (ii).  The middle rectangle is commutative by (iii).  This proves the initial case $n=1$.

Suppose $n\geq 2$.   We consider the two interior faces of $G$, say $F_1$ and $F_1'$, that appear first in the lists
\[\{H_1,\ldots,H_{j+n}\}\setminus \{A_1,\ldots,A_j\} \andspace \{H'_1,\ldots,H'_{k+n}\} \setminus \{A'_1,\ldots,A'_k\},\] respectively.  If $F_1 = F_1'$, then, similar to the case $n=1$, the two composites of $\phi$ are equal by (i)--(iii) and the induction hypothesis.

For the other case, suppose $F_1 \not= F_1'$.  Since $G$ has an underlying anchored graph, by \Cref{atomic-domain} $F_1$ and $F'_1$ do not intersect, except possibly for $t_{F_1}=s_{F'_1}$ or $t_{F'_1}=s_{F_1}$.   Similar to the $n=1$ case, by (i)--(iii) and the induction hypothesis, we are reduced to the case with $n=2$, $j=k=0$, the underlying anchored graph of $G$ as displayed below with each edge representing a directed path,
\[\begin{tikzpicture}[scale=1]
\node [plain] (s) {$s_G$}; \node [plain, right=1cm of s] (s1) {$s_{F_1}$}; 
\node [right=.5cm of s1] () {$F_1$};
\node [plain, right=1.5cm of s1] (t1) {$t_{F_1}$};
\node [plain, right=1cm of t1] (s2) {$s_{F'_1}$};
\node [right=.5cm of s2] () {$F'_1$};
\node [plain, right=1.5cm of s2] (t2) {$t_{F'_1}$};
\node [plain, right=1cm of t2] (t) {$t_{G}$};
\draw [arrow] (s) to node{$Q_1$} (s1); 
\draw [arrow, bend left=30] (s1) to node{$\dom_{F_1}$} (t1); 
\draw [arrow, bend right=30] (s1) to node[swap]{$\codom_{F_1}$} (t1);
\draw [arrow] (t1) to node{$Q_2$} (s2); 
\draw [arrow, bend left=30] (s2) to node{$\dom_{F'_1}$} (t2); 
\draw [arrow, bend right=30] (s2) to node[swap]{$\codom_{F'_1}$} (t2);
\draw [arrow] (t2) to node{$Q_3$} (t);
\end{tikzpicture}\]
and
\[\begin{split}
(\dom_G) &= b''\bigl((Q_1), (\dom_{F_1}), (Q_2), (\dom_{F'_1}), (Q_3)\bigr),\\
(\codom_G) &= b''\bigl((Q_1), (\codom_{F_1}), (Q_2), (\codom_{F'_1}), (Q_3)\bigr)
\end{split}\]
for some bracketing $b''$.  In this case, the equality of the two composites of $\phi$ follows from the bicategory axioms \eqref{hom-category-axioms}, \eqref{bicat-c-id}, and \eqref{middle-four}.
\end{proof}

As a corollary, we obtain essentially Power's pasting theorem \cite{power} for $2$-categories.
\begin{corollary}\label{2cat-pasting-theorem}
Suppose $\phi$ is a $G$-diagram in a $2$-category for some anchored graph $G$.  Then the composites of $\phi$ with respect to any two pasting scheme presentations of $G$ are equal.
\end{corollary}
\begin{proof}
The proof above restricts to a proof in the $2$-category case because in a $2$-category the associator is the identity natural transformation.  Therefore, the bracketings do not matter at all, and no associativity graphs are needed.  
\end{proof}




\end{document}